\documentclass[11pt]{amsart}
\usepackage{ulem}  
\usepackage{color} 

\theoremstyle{plain}
\newtheorem{thm}{Theorem}[section]

\newtheorem{lem}[thm]{Lemma}
\newtheorem{cor}[thm]{Corollary}
\newtheorem{prop}[thm]{Proposition}
\theoremstyle{definition}
\newtheorem{rmk}[thm]{Remark}

\numberwithin{equation}{section}

\newcommand{\ga}[2]{\begin{gather}\label{#1}#2 \end{gather}}

\newcommand{\sO}{{\mathcal O}}

\newcommand{\sS}{{\mathcal S}}

\newcommand{\A}{{\mathbb A}}

\newcommand{\C}{{\mathbb C}}

\newcommand{\F}{{\mathbb F}}

\renewcommand{\P}{{\mathbb P}}
\newcommand{\Q}{{\mathbb Q}}

\newcommand{\Z}{{\mathbb Z}}

\title [Lefschetz theorems]{Survey on some aspects of Lefschetz theorems in algebraic geometry}
\author{H\'el\`ene Esnault } 
\address{Freie Universit\"at Berlin, Arnimallee 3, 14195, Berlin,  Germany}
\email{esnault@math.fu-berlin.de}

\thanks{Supported by  the Einstein program}

\date{ November 27, 2016}
\begin{document}
\begin{abstract}
We survey classical material around Lefschetz theorems for fundamental groups, and show the relation to parts of Deligne's program in Weil II. 
\end{abstract}
\maketitle

\section{Classical notions}
\noindent
 Henri  Poincar\'e  (1854-1912) in \cite{Poi95} formalised the notion of {\it fundamental group} of a connected topological space $X$. It had appeared earlier on, notably in the work of Bernhard Riemann (1826-1866) (\cite{Rie51}, \cite{Rie57}) in the shape of multi-valued functions.\\[.1cm]
Fixing  a base point $x\in X$, then $\pi^{\rm top}_1(X,x)$ is first the set of homotopy classes of loops centered at $x$. It has a {\it group structure} by composing loops centered at $x$. It is a {\it topological invariant}, i.e. depends only on the homeomorphism type of $X$.  It is {\it functorial}: if $f: Y\to X$ is a continuous map, and $y\in Y$, then $f$ induces a homomorphism  $f_*: \pi_1^{\rm top}(Y,y)\to \pi_1^{\rm top}(X, f(y))$  of  groups.
 \\[.1cm]
If $X$ is locally contractible, for example if $X$ is a connected complex analytic manifold,
 its fundamental group  determines its topological coverings as follows: fixing $x$, there is a {\it universal covering $X_x$, together with a covering map $\pi: X_x\to X$, and a lift $\tilde{x}$ of $x$ on $X_x$,  such that $\pi_1^{\rm top}(X,x)$ is identified with ${\rm Aut}(X_x/X)$.}  More precisely, $\pi_1^{-1}(x)$ is a set, which is in bijection with $\pi_1^{\rm top}(X, x)$, sending the neutral element of the group to $\tilde{x}$.  In fact the universal covering exists under weaker local assumptions, which we do not discuss, as we only consider analytic and algebraic varieties in this note. 
 \\[.1cm]
Let us assume $X$ is a smooth projective algebraic  curve over $\C$, that is $X(\C)$ is a Riemann surface. By abuse of notations,  we write $ \pi_1^{\rm top}(X,x)$ instead of $ \pi_1^{\rm top}(X(\C),x)$.
 Then 
$\pi_1^{\rm top}(X,x)=0$ for $\P^1$, the Riemann sphere, that is if the genus $g$ of $X$ is $0$,  it is equal to $\Z^2$ if $g=1$ and else for $g\ge 2$, it is spanned by $2g$ generators $\alpha_i, \beta_i, i=1,\ldots, g$ with one relation $\prod_{i=1}^g [\alpha_i, \beta_i]=1$. So it is nearly a free group. In fact, for any choice of $s$ points $a_1,\ldots, a_s$ of $X(\C)$ different from $x$,  $s\ge 1$,  $\pi_1^{\rm top}(X\setminus \{a_1,\ldots, a_s\}, x)$ is free, is  spanned by  $\alpha_i, \beta_i, \gamma_1,\ldots, \gamma_s$, 
with one relation $\prod_{i=1}^g [\alpha_i, \beta_i]\prod_{j=1}^s \gamma_j=1$ (\cite{Hat02}).
 For $s=1$, the map
 $\pi_1^{\rm top}(X\setminus a, x) \to \pi_1^{\rm top}(X, x)$ is surjective and yields the presentation.
  \\[.1cm]
More generally, for any non-trivial Zariski open subvariety $U\hookrightarrow X$ containing $x$, 
the homomorphism $\pi_1^{\rm top}(U, x) \to \pi_1^{\rm top}(X, x)$ is always surjective, as we see taking loops and moving them via homotopies inside of $U$. The kernel in general is more complicated, but is spanned by loops around the divisor at infinity. \\[.1cm]
If $X$ has dimension $\ge 2$, then $\pi_1^{\rm top}(X,x)$  is far from being  free. A natural question is how to compute it.  This is the content of the Lefschetz  (Salomon Lefschetz (1884-1972)) theorems for the fundamental group. 
\begin{thm}[Lefschetz theorems] \label{thm:top_lef}
Let $X$ be a smooth connected projective variety defined over $\C$. Let $Y\to X$ be a smooth hyperplane section. Let $x\in Y\subset X$. Then the homomorphism of groups $\pi_1^{\rm top}(Y,x)\to \pi_1^{\rm top}(X,x)$ 
is 
\begin{itemize}
\item[1)] surjective if $Y$ has dimension $1$;
\item[2)] an isomorphism if $Y$ has dimension $\ge 2$.

\end{itemize}
\end{thm}
In particular 
\begin{cor} \label{cor:f_presented}
$\pi_1^{\rm top}(X,x)$ is a  {\it finitely presented} group.
\end{cor}
\noindent
In fact, both the theorem and its  corollary remain true for $\pi_1^{\rm top}(U,x)$, where $U$ is any non-trivial Zariski open subvariety in $X$. One takes $U\hookrightarrow X$ to be a good compactification,  that is $X$ is smooth projective such that $X\setminus U$ is a strict normal crossing divisor (strict meaning that all components are smooth). Then $Y$ in the theorem is replaced by the intersection $V=Y\cap U$, with the additional assumption that $Y$ is in good position with respect to $X\setminus U$.
\\[.1cm]
 In his proof  in  \cite{Lef24}, Lefschetz introduces the notion of {\it Lefschetz pencil}: one moves $Y$ in a one parameter family $Y_t, t\in \P^1$. For a good family, all fibres but finitely many of them are smooth.  His proof was not complete. In the \'etale context, it was proven only in \cite{SGA2}. Theorem~\ref{thm:top_lef} was proven in \cite{Bot59} using vector bundles: 
 let  $Y$ be a section of the bundle $\sO_X(Y)$. Bott uses the hermitian metric $h$ on $\sO_X(Y)$  to define the function $\varphi=1/(2\pi i)  \bar \partial \partial {\rm log} h(s)$.  Then he  proves that $X(\C)$ is obtained from $Y(\C)$ by attaching finitely many cells of dimension $\ge {\rm dim} (X)$ by doing Morse theory with $\varphi$. 

 \section{Galois theory}
 \noindent
 Let $K$ be a field, $\iota: K\hookrightarrow \bar K$ be a fixed separable closure.  One defines the group  ${\rm Aut}(\bar K/K)$ of automorphisms of $\bar K$ over $K$, endowed with its natural profinite topology.  This is `the' Galois group of $K$ associated to $\iota$. The main  theorem of Galois theory says that there is an  equivalence of categories  $\{$closed subgroups of  ${\rm Aut}(\bar K/K) \} \to \{$extensions $K\subset L \subset \bar K\}$ via $L=\bar K^H$ (\cite{Mil96}, \cite{Sza09}).  \\[.2cm]
 For $X$ a smooth connected variety defined over $\C$,   Grothendieck's key idea was to reinterpret $\pi_1^{\rm top}(X,x)$ as follows. One defines the category of {\it topological covers} ${\rm TopCov}(X)$. 
 The objects are maps $\pi: Y\to X$ where $Y$ is an Hausdorff topological space, locally homeomorphic to $X$ via $\pi$, or equivalently, $Y$ is an analytical space, locally biholomorphic to $X$ via $\pi$ (\cite[Thm.~4.6]{For81}). The maps are over $X$.
 The point $x$ yields a {\it fiber functor} $\omega_x: {\rm TopCov}(X)\to {\rm Sets}, \ (\pi: Y\to X(\C) ) \mapsto  \pi^{-1}(x)$. This means that   $\omega_x$  is {\it faithful}, that is  ${\rm Hom}(A,B)\to  {\rm Hom}(\omega(A), \omega(B))$ is injective. \\[.1cm]
 A unified presentation of Poincar\'e and Galois theories is as follows.
 \begin{thm}
 \begin{itemize}
 \item[1)]  ${\rm Aut}(\omega_x)=\pi_1^{\rm top}(X,x)$;
 \item[2)]  $\omega_x$ yields an equivalence of categories $${\rm TopCov}(X) \xrightarrow{\omega_x} {\rm Rep}_{\rm Sets}( \pi_1^{\rm top}(X,x)).$$
 \item[3)] The universal cover $X_x$ corresponds to the representation of $\pi_1^{\rm top}(X,x)$ by translation on itself. 
 \item[4)] A change of $x$ yields equivalent fibre functors $\omega_x$,  isomorphic $\pi_1^{\rm top}(X,x)$ and isomorphic $X_x$ over $X$. The equivalence and isomorphisms are not canonical.
 \end{itemize}
 \end{thm}
 \noindent
 In this language, 1) uses the  universal cover and the identification $\pi_1^{\rm top}(X,x)={\rm Aut}(X_x/X)$. \\[.1cm]
 One can interpret Galois theory in the same way. The embedding  $\iota: K\to \bar K$ corresponds  {\it both} to $x\to X$ and to $X_x \to X$. One defines ${\rm FinExt}(K)$ to be the category of {\it finite} separable $K$-algebra extensions  $K\subset L$.  Then $\iota$ defines a fibre functor $\omega_{\iota}: {\rm FinExt}(K) \to {\rm FinSets},  (K\hookrightarrow L ) \mapsto L\otimes_{\iota} \bar K$, the latter understood as a finite set indexing the split  $\bar K$-algebra $L\otimes_{\iota} \bar K$, and ${\rm FinSets}$ being now the category of finite sets.
 One defines $\pi_1(K, \iota)={\rm Aut}(\omega_\iota)$. It is  a {\it profinite} group.
 \begin{thm}[Galois theory revisited]
 \begin{itemize}
 \item[1)] $\omega_\iota$ yields an equivalence of categories $${\rm FinExt}(K) 
 \xrightarrow{\omega_\iota} {\rm Rep}_{\rm FinSets}( \pi_1(K , \iota)).$$
\item[2)] This equivalence extends to the category of Ind-extensions ${\rm Ext}(K)$ yielding  the functor $\omega_{\iota}: {\rm Ext}(K)\to {\rm Sets}$.  The functor 
$\omega_\iota$ yields an equivalence of categories $${\rm Ext}(K) \xrightarrow{\omega_\iota} {\rm ContRep}_{\rm Sets}( \pi_1(K ,\iota)).$$
\item[3)] $\iota: K\hookrightarrow \bar K$ corresponds to the continuous representation of 
$\pi_1(K,\iota)$ by translation on itself. 
\item[4)] A change of separable closure $\iota$ yields equivalent fibre functors $\omega_{\iota}$, isomorphic $\pi_1(K, \iota)$, simply called {\it  the Galois group } $G_K$ of $K$. The equivalence and the isomorphism are not canonical. 
 \end{itemize}
 \end{thm}
 \noindent
 To understand $\pi_1(K, \iota)$ for number fields  is the central  topic of one branch of number theory. For example,  the inverse Galois problem is the question whether or not $G_{\Q}$ can be as large as thinkable, that is whether or not any finite group is a quotient of $G_{\Q}$.  In these notes, we shall take for granted the knowledge of these groups. The main focus shall be on {\it finite} fields $k$.  For these, Galois theory due to Galois (!) shows that $G_{\F_q}=\varprojlim_n \Z/n=:\widehat{ \Z}$, where $\widehat{\Z}$ is topologically generated by the arithmetic Frobenius $\bar k\to \bar k, \lambda \mapsto \lambda^q$. 
 \section{\'Etale fundamental group \cite{SGA1}} \label{s:pi}
 \noindent
 This is the notion which unifies the topological fundamental group and Galois theory. 
 Let $X$ be a connected normal  (geometrically unibranch is enough) locally noetherian scheme. In \cite{SGA3}, it is suggested that one can  enlarge the category $\rm{\acute{E}t}(X)$ of pro-finite \'etale covers to discrete covers. It is important when one drops the normality assumption on $X$, and still requests  to have $\ell$-adic sheaves as representations of a (the right one)  fundamental group. A general theory of pro\'etale fundamental groups has been defined by Scholze \cite{Sch13},  and Bhatt-Scholze \cite{BS15}, but we won't discuss this, as we focus on Lefschetz theorems, and for those we need the \'etale fundamental group as defined in \cite{SGA1}.\\[.1cm]
   The category of finite \'etale covers ${\rm Fin\acute{E}t}(X)$   is the category of $\pi: Y\to X$ which are of  finite presentation, finite flat and unramified, or equivalently  of finite presentation,  finite smooth and unramified.  \\[.1cm]
The other basic data consist of a geometric point $x\in X$,   in fact a point in a separably closed field is enough.  Indeed, if $x$ is a point with separably closed residue field, then up to isomorphism there is only one geometric point  $\tilde x$ above it, which is algebraic, and the  fibre functors  
 $\omega_{\tilde x}: {\rm Fin \acute{E}t}(X) \to {\rm FinSets},( \pi: Y\to X) \mapsto \pi^{-1}(\tilde x)$
 associated to those geometric points are the same (not only isomorphic). \\[.1cm]
  So the construction explained now depends only on the point in a separable closure. But a geometric point enables one to take non-algebraic points, this gives more freedom as we shall see. \  
  The functor  $\omega_x: {\rm Fin \acute{E}t}(X) \to {\rm FinSets},( \pi: Y\to X) \mapsto \pi^{-1}(x)$, the latter understood as a finite set indexing the split algebra $\pi^{-1}(x)$ over $x$, is a fibre functor. 
One defines the {\it \'etale fundamental group of $X$ based at $x$} as $$\pi_1(X, x)={\rm Aut}(\omega_x).$$ It is  a {\it profinite} group, thus  in particular has a topology.  \begin{thm}[Grothendieck \cite{SGA1}]
 \begin{itemize}
 \item[1)] $\omega_x$ yields an equivalence of categories $${\rm Fin\acute{E}t}(X) \xrightarrow{\omega_x} {\rm Rep}_{\rm FinSets}( \pi_1(X ,x)).$$
\item[2)] This equivalence extends to the category of pro-finite \'etale covers  ${\rm \acute{E}t}(X)$ yielding $\omega_{x}: {\rm \acute{E}t}(X)\to {\rm Sets}$. 
$\omega_x$ yields an equivalence of categories $${\rm\acute{ E}t}(X) \xrightarrow{\omega_x} {\rm ContRep}_{\rm Sets}( \pi_1(X , x)).$$
\item[3)]  The continuous representation of $\pi_1(X,x)$ acting by translation on itself corresponds to $X_x\to X$, called the universal cover centered at $x$. 
\item[4)] A change of  $x$   yields equivalent $\omega_x$, isomorphic $\pi_1(X, x)$ and isomorphic $X_x$ over $X$.  The equivalence and the isomorphisms are not canonical.
 \end{itemize}
 \end{thm}
 \section{Comparison}
 \noindent
 We saw now the formal analogy between topological fundamental groups, Galois groups and \'etale fundamental groups. 
 We have to see the geometric relation. 
 \begin{thm}[Riemann existence theorem, \cite{Rie51}, \cite{Rie57}] \label{thm:RET}
 Let $X$ be a smooth variety over $\C$. Then a finite \'etale cover $\pi_{\C}: Y_{\C} \to X(\C)$ is the complex points  $\pi(\C): Y(\C)\to X(\C)$ of a uniquely defined finite \'etale cover $\pi: Y\to X$.
 \end{thm}
 \begin{cor}[Grothendieck, \cite{SGA1}] \label{cor:profinite}
 The \'etale fundamental group  $\pi_1(X,x)$ is the profinite completion of the topological fundamental group $\pi_1^{\rm top}(X, x)$, where $x\in X(\C)$.
 \end{cor}
\noindent
 In particular, using localization and the Lefschetz theorems, one concludes
 \begin{cor} \label{cor:ft}
 Let $X$ be a smooth variety over $\C$.  Then $\pi_1(X,x)$ is topologically of finite type, that is there is a finite type subgroup of $\pi_1(X,x)$  which is dense for the profinite topology.
 
 \end{cor}

 \section{Homotopy sequence and base change}

\noindent
 However, it does not shed light on the structure of $\pi_1(X,x)$ in general.\\[.1cm]
 In the sequel we shall always assume $X$  to be connected,  locally of finite type over a field $k$ and to be  geometrically connected over $k$. The last condition is equivalent to  $k$ being equal to its  algebraic closure in $\Gamma(X, \sO_X)$.
 Given the geometric point $x\in X$, defining the algebraic closure $\iota: k\to \bar k\subset k(x)$ of $k$ in the residue field $k(x)$ of $x$, 
 the functors $${\rm Ext}(k) \to {\rm \acute{E}t}(X),  \ (k\subset \ell)  \mapsto (X_\ell\to X)$$ and $${\rm \acute{E}t}(X) \to {\rm \acute{E}t}(X_{k(x)}),  \ (\pi: Y\to X)\mapsto  (\pi_{k(x)}: Y_{k(x)}\to X_{k(x)})$$
  define the {\it  homotopy sequence}  of continuous homomorphisms
 \ga{1}{ 1\to \pi_1(X_{k(x)}, x)\to \pi_1(X, x) \to \pi_1(k, \iota)\to 1.}
 \begin{thm}[Grothendieck's homotopy exact sequence, \cite{SGA1}]  \label{thm:HES}
 The homotopy sequence \eqref{1} is exact.
 \end{thm}
 \noindent
Surjectivity on the right means precisely this: 
 ${\rm Ext}(k) \to {\rm \acute{E}t}(X)$ is fully faithful, and any intermediate \'etale cover $X_\ell \to Y\to X$ comes from $k\subset \ell'\subset \ell$, with $(Y\to X)=(X_{\ell'} \to X)$. Injectivity on the left means that 
  any finite \'etale cover of $X_{k(x)}$ comes from some finite \'etale cover of $X$ (not necessarily geometrically connected) by taking a factor, and exactness in the middle means that 
  given $Y\to X$ finite \'etale, such that  $Y_{k(x)} \to X_{k(x)}$  is completely split, 
 then there is a $\ell \in {\rm Ext}(k)$  such that $Y\to X$ is $X_\ell \to X$. See \cite[~49.14,~49.4.3,~49.4.5]{StacksProject} \\[.2cm]
  There is a more general homotopy sequence: one replaces $X\to {\rm Spec}( k)$ by $f: X\to S$ a {\it proper separable} morphism (\cite[Exp.~X, Defn.~1.1]{SGA1}) of locally noetherian schemes. Separable means that $f$ is flat, and  all fibres $X_s$ are separable, i.e. reduced after all field extensions $X_s\otimes_s K$.   Let $s=f(x)$. Then  analogously defined functors   yield the sequence
  \ga{2}{ \pi_1(X_{s}, x)\to \pi_1(X, x) \to \pi_1(S,s )\to 1.}
  \begin{thm}[Grothendieck's second homotopy exact sequence, \cite{SGA1}]  \label{thm:HES2} If $f_*\sO_X=\sO_S$, 
  the homotopy sequence \eqref{2}  is exact.
  \end{thm}
  \noindent
 Let $\iota_K: k\hookrightarrow K$ be an embedding in an algebraically closed field, defining $\iota: k\hookrightarrow \bar k \hookrightarrow K$. This defines the functor ${\rm \acute{E}t}(X_{\bar k})\to  {\rm \acute{E}t}(X_{K}),  \ (Y\to X_{\bar k} ) \to (Y_K\to X_K)$. Let us denote by $x_K$ a $K$-point of $X$, and by $x_{\bar k}$ the induced $\bar k$-point by the map $X_K\to X_{\bar k}$. 
 \begin{prop} \label{prop:basechange}
 This functor induces a surjective homomorphism $\pi_1(X_K, \iota_K) \to  \pi_1(X_{\bar k}, \iota)$.
 If $X$ is proper, it is an isomorphism. 
 \end{prop}
 \noindent
 Surjectivity again amounts to showing that if  $Y$ is a connected scheme, locally of finite type over $\bar k$, then $Y_K$ is connected as well.  This is a local property, so one may assume that $Y={\rm Spec}(A)$ where $A$ is an affine $\bar k$-algebra, where $\bar k$ is algebraically closed in $A$. Then 
 $Y_K={\rm Spec}(K\otimes_{\bar k} A)$, and $K=K\otimes_{\bar k} \bar k$ is the algebraic closure of $K$ in $K\otimes_{\bar k} A$. 
 \\[.1cm]
 The homomorphism 
 $\pi_1(X_K, x_K)\to \pi_1(X,x)$ factors through $\pi_1(X_{\bar k}, x_{\bar k})$.  Thus 
 if $X$ is proper, 
   Theorem~\ref{thm:HES2} implies the second assertion of Proposition~\ref{prop:basechange}.
  \\[.2cm]
 Yet if $k$ has characteristic $p>0$,  {\it  injectivity it not true} in general. For example, 
 setting $X=\A^1$, and $ \bar k\hookrightarrow  \bar k[t]\hookrightarrow K$,  
 the Artin-Schreier cover  $x^p-x=st$ in $\mathbb{A}^2$  is not constant in $t$. (Example of Lang-Serre, \cite{SGA1}). 
If the homomorphism of Proposition~\ref{prop:basechange} was injective, the quotient 
$\Z/p$ of $\pi_1(X_K, \iota_K) $ defined by this example $\pi_K: Y_K\to \A^1_K$ would factor through $\pi_1(X_{\bar k}, \iota)$, 
so there would be an Artin-Schreier cover  $\pi: Y_{\bar k}\to \A^1_{\bar k}$ which pulls-back over $K$ to $\pi_K$, so 
$\pi_K$ would be constant.  Compare with Proposition~\ref{prop:basechangechar0} in characteristic $0$.
   \\[.1cm]
 In fact this is more general.
 \begin{lem}
 Let $C\subset \A^2$ be any smooth geometrically connected  curve, with $x\in C$, then the homomorphism $\pi_1(C,x)\to \pi_1(\A^2,x)$ is {\it never surjective}.  So there can't possibly be any Lefschetz type theorem for  $\pi_1(X,x)$ when $X$ is non-proper.
 \end{lem}
 \begin{proof}
 By Theorem~\ref{thm:HES} we may assume $k=\bar k$.  Let $f\in k[s,t]$ be the defining equation of $C$, $L$ be a line cutting $C$ transversally. Thus the restriction $f|_L:  L\to \A^1$ has degree $d \ge 1$, the degree of $f$, and is ramified  in $d$ points. Thus if $T\to \A^1 $ is a connected Artin-Schreier cover,  then $T\times_{\A^1} L$ is connected, thus as fortiori $T\times_{\A^1} \A^2$ is connected as well, and $T\times_{\A^1} \A^2\to \A^2$ induces  a $\Z/p$-quotient of $\pi_1(\A^2, x)$. However, $(T\times_{\A^1} \A^2) \times_{\A^2} C = (T\times_{\A^1} 0 ) \times_k C\to \A^2\times_{\A^2} C=C$ splits completely.  Thus the composite homomorphism 
 $\pi_1(C, x)\to \pi_1(\A^2, x) \to \Z/p$ is $0$.

 \end{proof}
 \noindent
 We remark however: Theorem~\ref{thm:HES2} enables one to compare $\pi_1(X_{\bar k})$ in characteristic $0$ and $p>0$. Indeed, assume $X$ is separable proper of finite type, defined over an algebraically closed characteristic $p>0$ field $\bar k$, and there is a model $X_R/R$, i.e. flat (thus proper), over a strictly henselian  ring $R$
 with residue field $\bar k$ and field of fractions $K$ of characteristic $0$. Then $X_R/R$ is separable.  Theorem~\ref{thm:HES2} refines to saying 
 \begin{prop} 
 $\pi_1(X_{\bar k}, x) \to \pi_1(X_R, x)$
 is an isomorphism. 
 \end{prop}
 \noindent
 This is a direct consequence of the equivalence of categories between ${\rm \acute{E} t}(X_{\bar k})$ and ${\rm \acute{E} t}(X_{R_n})$ 
 (\cite[Thm.~18.1.2]{EGA4}), and of the formal function theorem (\cite[Thm.~5.1.4]{EGA3}). 
Here $R_n=R/\langle \pi^n\rangle$, where $\pi$ is a uniformizer of $R$.\\[.2cm]

  \noindent
 This defines the {\it specialization homomorphism} $sp: \pi_1(X_{\bar K}, x_{\bar K})\to \pi_1(X_{\bar k}, x)$, if $x_{\bar K}$ specializes to $x$.
 \begin{thm}[Grothendieck's specialization theorem, \cite{SGA1}] \label{thm:sp}
 If $X_R/R$ is proper separable, then sp is surjective.
 
 \end{thm}
 \noindent
 We see that a  surjective specialization can not exist for non-proper varieties, e.g. over for $\A^1_R$, where $R$ is the ring of Witt vectors of an algebraically closed characteristic $p>0$ field.  But in the proper case, the \'etale fundamental group can not be larger than the one in characteristic $0$. 
 
 \section{Remarks on conjugate varieties}
 \noindent
  Let  $X$ be a complex  variety. Then $X$ is defined over a subfield $k\subset \C$ of finite  type over $\Q$. Let $k\subset K $ be any algebraically closed field, and $\bar k$ be the algebraic closure of $k$ in $K$.  Then Proposition~\ref{prop:basechange} is now much better behaved:
  \begin{prop} \label{prop:basechangechar0}
  The homomorphism $\pi_1(X_K, \iota_K) \to  \pi_1(X_{\bar k}, \iota)$ is an isomorphism.

  \end{prop}
  \begin{proof}
  Surjectivity comes from  Proposition~\ref{prop:basechange}.  On the other hand, any finite \'etale cover $\pi: Y\to X_K$ is defined over an affine algebra $R=k[S]$, say $\pi_R: Z \to X_R$, so that $\pi_R\otimes _R K =\pi$.  We choose  a complex embedding $k\hookrightarrow \C$, inducing the embedding
  $R\hookrightarrow \C[S]$.  Thus $(\pi_R)\otimes_k \C$ is a finite \'etale cover of $X_{\C} \times_{\C} S_{\C}$. As the topological fundamental group verifies the K\"unneth formula, one concludes that $Z\otimes_k {\C}$ is isomorphic over $S_{\C}$ to  $V\times_{\C} S_{\C}$, for some connected finite \'etale cover  $V\to X_{\C}$.  So $V$ is isomorphic over $\C$ to $Z_s\times_s \C$, and $Z\otimes_k \C$ is isomorphic over $S_{\C}$  to $(Z_s\times_k S)\otimes_k \C$  for any closed point $s \in S$.  Since this isomorphism 
  is defined over an affine $k$-scheme $T$ say, by specializing at a closed point of $T$ one obtains the splitting $Z=Z_s\times_k S$ over $S$.

  \end{proof}
  \noindent
However, Proposition~\ref{prop:basechangechar0} does not extend to the topological fundamental group. Indeed Serre (\cite{Ser64}) constructed an example of a $X$ together with a $\sigma \in {\rm Aut}(\C/K)$ such that the topological fundamental group $\pi_1^{\rm top}(X_{\sigma})$ of the conjugate variety $X_\sigma$ is not isomorphic to the  the topological fundamental group $\pi_1^{\rm top}(X)$ of $X$.  (We do not write the base points as they do not play any r\^ole).  \\[.1cm]

\section{Lefschetz theorems in the projective case and in the tame case}
\begin{thm}[Grothendieck's Lefschetz theorems, \cite{SGA2}] \label{thm:et_lef}
Let $X$ be a regular  geometrically connected projective variety defined over  a field $k$. Let $Y\to X$ be a regular hyperplane section. Let $x\in Y\subset X$ be a geometric point. Then the continuous homomorphism of groups $\pi_1(Y,x)\to \pi_1(X,x)$ 
is 
\begin{itemize}
\item[1)] surjective if $Y$ has dimension $1$;
\item[2)] an isomorphism if $Y$ has dimension $\ge 2$.

\end{itemize}
\end{thm}
\noindent
We see immediately a wealth of corollaries of this fundamental theorem.  Let us comment on a few of them.\\[.1cm]
The theorem implies in particular that $Y$ {\it  is geometrically connected}.  If $X$ and $Y$  are smooth over $k$,  then $Y_{\bar k}$ can't have several components as each of them would be ample, thus they would meet, and $Y_{\bar k}/\bar k$ could not be smooth.  But if we assume only regularity, this already is a subtle information.

\begin{cor}  \label{cor:ft_proj} If in addition, $X/k$ is smooth, 
 $\pi_1(X_{\bar k}, x)$ is topologically of finite type.

\end{cor}
\begin{proof} If $k$ has characteristic $0$, we just apply  Corollary~\ref{cor:profinite} together with Proposition~\ref{prop:basechange}. If $k$ has characteristic $p>0$, 
applying Theorem~\ref{thm:et_lef} we may assume that $X$ has dimension $1$. Then $X$ lifts to characteristic $0$, so we apply Theorem~\ref{thm:sp}, and then Proposition~\ref{prop:basechangechar0} to reduce the problem to $k=\C$, then Corollary~\ref{cor:profinite} to reduce to the explicit topological computation.

\end{proof}
\noindent
One has the notion of {\it tame} fundamental group.  Recall that a finite  extension  $R\hookrightarrow S$ of discretely valued rings  is {\it tame} if the ramification index is not divisible by $p$, the residue characteristic, and if the residue field extension is separable (\cite{Ser62}). \\[.1cm]
One has two viewpoints to define tame coverings. If $X$ has a good compactification $X\hookrightarrow \bar X$ with a strict normal crossings divisor at infinity, then  a finite \'etale cover  $\pi: Y\to X$ is said to be tame if, for $\bar Y$ the normalization of $\bar  X$ in the field of rational functions on $Y$, and all codimension $1$ points $y$ on $\bar Y$, with image $x$ in $\bar X$, a codimension $1$ point on $\bar X$,  the extension $\sO_{\bar X,x}\hookrightarrow \sO_{\bar Y, y}$ is tame. \\[.1cm]
Another viewpoint is to say that a finite \'etale cover $\pi: Y\to X$ is tame if and only if it is after restriction to all smooth curves $C\to X$. \\[.1cm]
That those two definitions are equivalent is a theorem of Kerz-Schmidt \cite{KS10}. It enables one to define tame covers via the curve criterion without having a good compactification at disposal.  \\[.1cm]
This defines the {\it tame fundamental group} as a continuous quotient  
$\pi_1(X,x)\to \pi_1^t(X,x)$. 
By definition, the tame quotient factors 
$$\pi_1(X,x)\to \pi_1^t(X,x)\to\pi_1(k, \iota)$$  from \eqref{1}, which is surjective if $X$ is geometrically connected. 
\begin{thm}[Tame Lefschetz theorems, \cite{EKin15}] \label{thm:EK}
Let $X\hookrightarrow \bar X$ be a good regular projective compactification of a regular quasi-projective scheme $X$ defined over a field, that is $D=\bar X\setminus X$ is a  normal crossings divisor with regular components. Let $\bar Y$ be a regular hyperplane section which intersects $D$ transversally. Set $Y=\bar Y\setminus D \cap \bar Y$. 
Then the continuous homomorphism of groups $\pi^t_1(Y,x)\to \pi^t_1(X,x)$ 
is 
\begin{itemize}
\item[1)] surjective if $Y$ has dimension $1$;
\item[2)] an isomorphism if $Y$ has dimension $\ge 2$.

\end{itemize}
\end{thm}
\noindent
Again we see that this implies in particular that $Y$ and $X$ have the same field of constants. 
\begin{cor}  \label{cor:ft_qproj} If in addition $X/k$ is smooth,  then
 $\pi_1^t(X_{\bar k}, x)$ is topologically of finite type.

\end{cor}
\begin{proof}
Theorem~\ref{thm:EK} enables one to assume ${\rm dim}(X)=1$.  Then one argues as for Corollary~\ref{cor:ft_proj}, applying the surjective  specialization homomorphism of Mme Raynaud (see \cite[XIII, Cor.~2.12]{SGA1}) for the tame fundamental group for $X_R\subset \bar X_R$ a relative normal crossings divisor compactification of the curve $X_R$ over $R$. 
\end{proof}
\section{Deligne's  $\ell$-adic conjectures in Weil II}
\noindent
In Weil II,   \cite[Conj.1.2.10]{Del80},   Deligne  conjectured that if $X$ is a normal connected scheme of finite type over a finite field, and $V$ is an irreducible  lisse 
$\bar{\Q}_\ell$ sheaf of rank $r$ over $X$,  with finite determinant, then 
\begin{itemize}
 \item[(i)] $V$ has weight $0$, 
\item[(ii)] there is a number field $E(V)\subset \bar{ \Q}_\ell$
containing all the coefficients of the local characteristic  polynomials $f_V(x)(t)={\rm
  det}(1-tF_x|_{V_x})$, where $x$ runs through the closed points of $X$ and $F_x$ is the
geometric Frobenius at the point $x$, 
\item[(iii)] $V$ admits $\ell'$-companions for all prime
numbers $\ell'\neq p$.  
\end{itemize}
The last point means the following. Given a field isomorphism $\sigma: \bar \Q_\ell \to \bar \Q_{\ell'}$ for prime numbers $\ell, \ell'$ (possibly $\ell=\ell'$) different from $p$, there is a  $\bar \Q_{\ell'}$-lisse sheaf $V_{\sigma}$ such that $$\sigma f_V=f_{V_{\sigma}}.$$
(Then automatically, $V_{\sigma}$ is irreducible and has finite determinant as well.)\\[.1cm]
As an application of his Langlands correspondence for ${\rm GL}_r$, 
Lafforgue  proved  (i), (ii), (iii) for $X$ a smooth curve \cite{Laf02}. 
In order to deduce (i) on $X$ smooth of higher dimension, one needs a Lefschetz type theorem. 

\begin{thm}[Wiesend \cite{Wie06}, \cite{Wie08}, Deligne \cite{Del12}, Drinfeld \cite{Dri12}] \label{thm:curve}
Let $X$ be a smooth variety over $\F_q$ and $V$ be an irreducible $\bar \Q_\ell$-lisse  (or Weil) sheaf. Then there is a smooth curve $C \to  X$ such that $V|_C$ is irreducible. One can request $C$ to pass through a finite number of closed points $x\in X$ with the same residue field $k(x)$. 

\end{thm}
\noindent 
There is a mistake in \cite{Laf02} on this point. 
The  Lefschetz theorem is in fact weaker than claimed in {\it loc.cit.}, the curve depends  on $V$, and is not good for all $V$ at the same time. 
\begin{proof} The sheaf
$V$ corresponds to a representation $\rho: \pi_1(X,\bar x)\to GL(r, R)$ where $R\supset \Z_\ell$ is a finite extension.  Let $\frak{m}\subset R$ be its maximal ideal.  One defines $H_1$ as the kernel of 
$\pi_1(X, \bar x)\to GL(r, R)\to GL(r, R/\frak{m})$. Let $X_1\to X$ be the Galois cover such that $H_1= \pi_1(X_1,\bar x)$. One defines  $H_2$ to be the intersection of the kernels of all homomorphisms $ H_1\to \Z/\ell$.  As $H_1(X_1, \Z/\ell)$  is finite, $H_2\hookrightarrow H_1$ is of finite index. Then $H_2$ is normal in $\pi_1(X, \bar x)$. This defines  the covers $X_2\to X_1\to X$, where $X_2\to X$ is Galois and $H_2= \pi_1(X_2, \bar x)$.   In addition,  any continuous homomorphism  $ K\to \rho(\pi_1(X, \bar x))$ from a profinite group $K$  is surjective 
if and only if its quotient
$K\to  \pi_1(X, \bar x)/\pi_1(X_2, \bar x)= \rho( \pi_1(X, \bar x))/\rho(\pi_1(X_2, \bar x) )$ is surjective.  \\[.1cm] 
Then one needs a curve $C$ passing through $x$ such that $\pi_1(C, \bar x)\to 
\pi_1(X, \bar x)/\pi_1(X_2, \bar x)$ is surjective. 
To this: one may apply Hilbert irreducibility (Drinfeld), see also \cite{EK12},  or Bertini \`a  la Jouanolou (Deligne). \\[.1cm] 
On the latter: one may assume $X$  affine (so $X_2$ affine as well). Then take an affine embedding  $X\hookrightarrow \A^N$, and define the Grassmannian of lines in $\A^N$. Bertini implies there is a non-empty open on $\bar \F_q$ on which the pull-back on $X_2\otimes  \bar \F_q$ of any closed point is connected and smooth. Making the open smaller, it is defined over $\F_q$. It yields the result.
 In the construction, one  may also fix first a finite number of closed points. The curves one obtains in  this way are defined over $\F_{q^n}$ for a certain $n$ as our open might perhaps have no $\F_q$-rational point. \\[.1cm]
On the former: one may assume $X$ affine and consider a Noether normalisation $\nu : X\to \A^d$ which is generically \'etale.  The points $x_i$ map to $y_i$  (in fact one may even assume $\nu$  to be  \'etale at those points). Then take a linear projection $\A^d\to \A^1$ and consider $X_{2, k(\A^1)}\to X_{k(\A^1)} \to \A^{d-1}_{k(\A^1)}$. Hilbert irreducibility implies there are closed $k(\A^1)$-points of $\A^{d-1}_{k(\A^1)}$,  thus with value in $k(\Gamma_i)$ for finite covers $\Gamma_i \to \A^1$, which does not split in $
X_{2, k(\A^1)}$, the image of which in $X_{k(\A^1)} $ specialises to $x_i$.

\end{proof}

\noindent
 Using Lafforgue's results, Deligne showed (ii) 
in 2007 \cite{Del12}. He first proves it on curves. Lafforgue  shows that finitely many Frobenii of closed points determine the number field containing all eigenvalues of closed points. Deligne makes the bound effective, depending on the ramification of the sheaf and the genus of the curve. Then given a closed point of high degree, he shows the existence of a  curve with small bound which passes through this point. \\[.1cm]
Using (ii) and ideas of Wiesend, Drinfeld  \cite{Dri12} showed
(iii) in 2011,  assuming in addition $X$ to be smooth.  

\noindent
To reduce the higher dimensional case  to curves, starting with $V$, Drinfeld shows that a system of eigenvalues for all closed points comes from an $\sigma(\sO_{E_{\ell}})$-adic sheaf, where $E_\ell$ is the finite extension of $\Q_\ell$ which contains the monodromy ring of $V$,  if and only if it does on curves and has tame ramification on the finite \'etale cover  on which $V$ has tame ramification.\\[.1cm]
Let $E(V)$ be Deligne's number field for $V$ irreducible with finite determinant, on $X$ {\it smooth} connected over $\F_q$.
\begin{thm}[Lefschetz for $E(V)$, \cite{EK12}]  \label{thm:E(V)} Let $X$ be smooth over $\F_q$.
\begin{enumerate}
\item[1)] For $\emptyset \neq  U \subset X$ open,  $E(V|_{U})=E(V)$.
\item[2)] Assume $X$ has a good compactification $X\hookrightarrow \bar X$. Let $C\hookrightarrow X$ be a smooth curve passing through $x$ such that 
$\pi_1^t(C, \bar x)\to  \pi_1^t(X,\bar x)$  is surjective
(Theorem~\ref{thm:EK}, 1)). Then for all tame $V$ irreducible with finite determinant, $E(V|_C)=E(V)$.

\end{enumerate}
\end{thm}
\begin{proof}
Ad 1): One has an obvious injection $E(V|_{U})\hookrightarrow E(V)$. To show surjectivity, let $\sigma \in {\rm Aut} (E(V)'/E(V|_U))$, where $E(V)'/E(V|_U)$ is the Galois closure of $E(V)/E(V|_U)$ in $\bar \Q_\ell$.  Then $f(V_{\sigma}|_U)=f(V|_U)$ so by \v{C}ebotarev's density theorem, one has  $V_{\sigma}|_U=V|_U$. From the surjectivity 
of $\pi_1(U, \bar x)\to \pi_1(X,\bar x)$
 one obtains $V_\sigma=V$.\\[.1cm]
Ad 2):  Same proof as in 1): take $\sigma \in {\rm Aut}(E(V)/E(V|_C))$. Then 
$f(V_{\sigma}|_C)=f(V|_C)$, thus by the Lefschetz theorem $V_{\sigma}=V$. \\[.1cm]

\end{proof}
\noindent
Let $X$ be a  geometrically connected scheme of finite type over $\F_q$,  $\alpha: X'\to X$ be a finite \'etale cover. One says that a $\bar \Q_\ell$-lisse sheaf  $V$ has ramification bounded by $\alpha$ if $\alpha^* (V)$ is tame (\cite{Dri12}). 
This means that for any smooth curve $C$ mapping to $X'$, the pullback of $V$ to $C$ is tame in the usual sense.
If $X$ is geometrically unibranch, so is $X'$, and $V$ is defined by a representation $\rho: \pi_1(X)\to GL(r, R)$ of the fundamental group 
 where $R\supset \Z_\ell$ is a finite extension of discrete valuation rings.  Then $V$ has ramification bounded by  any $\alpha$ such that $\pi_1(X')\subset {\rm Ker} \big(\pi_1(X)\to GL(r, R)\to GL(r, R/2\ell)\big)$.
If $X$ is smooth, so is $X'$, and $\alpha^*(V)$ is tame amounts  to say that the induced representation of $\pi_1(X')$ factors through $\pi^t_1(X')$. Given a natural number $r$ and  given $\alpha$, one defines  $\sS(X, r, \alpha)$ to be  the set of isomorphism classes of irreducible  $\bar \Q_\ell$-lisse sheaves $V$ of rank $r$  on $X,$  such that $\alpha^*(V)$ is tame, modulo twist by a character of the Galois group of $\F_q$. \\[.1cm]
\noindent
Let $X$ be a smooth scheme of finite type over $\F_q$,  $X\hookrightarrow \bar X$ be a normal compactification, and $D$ be a Cartier divisor with support $\bar X\setminus X$. One says that a $\bar \Q_\ell$-lisse sheaf  $V$ has ramification bounded by $D$ if for any smooth curve $C$ mapping to $X$, 
with compactification $\bar C\to \bar X$, where $\bar C$ is smooth, 
the pullback $V_C$ of $V$ to $C$ has Swan conductor bounded above by $\bar C\times_{\bar X} D$ (\cite{EK12}). If $V$ has ramification bounded by $\alpha$, then also by $D$ 
for some Cartier divisor with support $\bar X\setminus X$ (\cite{EK12}).
 Given a natural number $r$ and  given $D$, one defines  $\sS(X, r, D)$ to be  the set of isomorphism classes of irreducible $\bar \Q_\ell$-sheaves of rank $r$ on $X$, of ramification bounded by $D$, modulo twist by a character of the Galois group of $\F_q$. \\[.1cm]
\begin{thm}[Deligne's finiteness, \cite{EK12}] \label{thm:finiteness}
On $X$ smooth of finite type over $\F_q$, with a fixed normal compactification $X\hookrightarrow \bar X$ and a fixed Cartier divisor $D$ with support $\bar X\setminus X$, the set $\sS(X, r, D)$ is finite.
\end{thm}

\noindent
Deligne's proof is very complicated.  It relies on the existence of companions, thus $X$ has to be smooth.  However, one can prove  the following variant of Deligne's finiteness theorem in a very simple way, without using the existence of the companions and the existence of the number field. 
\begin{thm}[\cite{Esn16}] \label{thm:E}
On $X$  geometrically unibranch over $\F_q$, with a fixed finite \'etale cover $\alpha: X'\to X$, the set $\sS(X,  r, \alpha)$ is finite.

\end{thm}
\noindent
The proof relies crucially on Theorem~\ref{thm:E(V)}. Indeed, on a good alteration $Y\to X'$, there is one curve $C\to Y$ such that $\pi_1^t(C)\to \pi_1^t(Y)$ is surjective. Then, for any natural number $s$, 
the set $\sS(Y, s, {\rm id})$ is recognized via restriction in $\sS(C, s, {\rm id})$, which is finite by \cite{Laf02}.
\begin{rmk}
If one were able to reprove Drinfeld's theorem in an easier way, one could this way reprove Deligne's 
theorem on the existence of the number field: indeed ${\rm Aut}(\bar \Q_\ell/\Q)$ then acts on the $\ell$-adic irreducible sheaves with bounded rank and ramification, and each object  $V$ has a finite orbit. So the stabilizer  $G$ of $V$ is of finite index in ${\rm Aut}(\bar \Q_\ell/\Q)$. 
This defines $E(V)=\bar \Q_\ell^G$,  a finite extension of $\Q$.
\end{rmk}
\noindent
The method used in the  proof of Theorem~\ref{thm:E} enables one to enhance the Lefschetz theorem for $E(V)$. 
\begin{thm}[Lefschetz for $E(V)$] Let $X$ be smooth of finite type over $\F_q$, 
let  $\alpha: X'\to X$ be a finite \'etale cover, let a natural number $r$ be given. 
  Then there is a smooth curve $ C \to X$, finite over its image, and a natural number $m>0$ such that $E(V|_C)$ contains $E(V\otimes \F_{q^m})$ for all $\bar \Q_\ell$-lisse sheaves $V$ with class in $\sS(X,r,\alpha)$.

\end{thm}

\section{Deligne's crystalline conjecture in Weil II} \label{s:crys}
\noindent
In  \cite[Conj.1.2.10]{Del80}, Deligne predicts crystalline companions, without giving a precise conjecture. It has been later  made precise by Crew \cite{Cre86}.
In the sequel, $X$ is a smooth geometrically connected variety of finite type  over $\F_q$.  We briefly recall the definition of the category of  $F$-overconvergent isocrystals. \\[.1cm]
Crystals are crystals in the crystalline site. They form a $W(\F_q)$-linear category. Its $\Q$-linearization is the category of  isocrystals.  It is a $K$-linear category, where $K$ is the field of fractions of $W(\F_q)$.
The Frobenius (here $x\mapsto x^q$) acts on the category.   An isocrystal $M$ is said to have an $m$-th Frobenius structure, or equivalently is an $F^m$-isocrystal, for some natural number $m\ge 1$,  if it is endowed with an isomorphism $F^{m*}M\cong M$ of isocrystals. 
An isocrystal with a Frobenius structure is  necessarily convergent. Convergent isocrystals are the isocrystals which are $F^\infty$-divisible. 
 Isocrystals with a Frobenius structure  are not necessarily overconvergent. Overconvergence   is an analytic property along the boundary of $X$, and concerns the  radius of convergence at infinity of $X$ of the underlying $p$-adic differential equation. Overconvergence 
  is defined  on isocrystals, whether or not they carry a Frobenius structure. 
  One defines the $\bar \Q_p$-linear category of $F$-{\it overconvergent isocrystals} as follows
 (see \cite[Section~1.1]{AE16}). One first considers the category of overconvergent isocrystals over $K$,  then $\bar \Q_p$-linearize it for a given algebraic closure $K\hookrightarrow \bar \Q_p$. In this $\bar \Q_p$-linear category, one defines the subcategory of 
   isocrystals  with an $F^m$-structure in this category, for some natural number $m\ge 1$. The morphisms respect all the structures. 
It is a $\bar \Q_p$-linear tannakian category.\\[.1cm]
 \noindent
 The analytic overconvergence condition is difficult to understand. However, 
 Kedlaya \cite{Ked07}  proved that an isocrystal with a Frobenius structure is overconvergent if and only if
  there is an alteration $ Y\to X$,  with $Y$ smooth and $Y\hookrightarrow \bar Y$ is a good compactification, such that the isocrystal $M$, pulled back to $Y$, has nilpotent residues at infinity. \\[.2cm]
The category of $F$-overconvergent isocrystals 
is believed to be the 'pendant' to the category of $\ell$-adic sheaves. 
More precisely, Deligne's conjecture can be interpreted as saying: \\[.1cm]
\begin{enumerate}
\item For $V$ an irreducible $\bar \Q_\ell$-sheaf with torsion determinant, and an abstract isomorphism of fields $\sigma: \bar \Q_\ell  \xrightarrow{\cong} \bar \Q_p$, there is an irreducible $F$-overconvergent isocrystal $M$ with torsion determinant, called $\sigma$-companion,  with the property: for any closed point $x$ of $X$, 
the characteristic polynomials  $f_V(x)\in \bar \Q_\ell[t]$ of   the geometric  $F_x$ acting on $V_x$ and the characteristic polynomial $f_M(x) \in \bar \Q_p[t]$ of the absolute Frobenius (still denoted by) $F_x$ acting on $M_x$ are the same via $\sigma$:  $\sigma f_V=f_M$. 
\item And vice-versa: for $M$ an
 irreducible $R$-overconvergent isocrystal  with torsion determinant, and an abstract isomorphism of fields $\sigma: \bar \Q_\ell  \xrightarrow{\cong} \bar \Q_p$, there is an irreducible $\bar \Q_\ell$-lisse sheaf $V$  with torsion determinant  with the property $\sigma f_V=f_M$. 
\end{enumerate}
\begin{thm}[Abe, crystalline companions on curves, \cite{Abe13}]
The whole strength of (1) and (2) is true on smooth curves. 
\end{thm}
A first step towards a Lefschetz theorem for $F$-overconvergent isocrystals is the following weak form of \v{C}ebotarev-density theorem.
\begin{thm}[Abe,  \v{C}ebotarev for $F$-overconvergent isocrystals, \cite{Abe13}]
If two $F$-overconvergent isocrystals have their eigenvalues of the local Frobenii equal, then their semi-simplification are isomorphic. 

\end{thm}
\noindent
One  has an analog of Theorem~\ref{thm:EK}.

\begin{thm}[Tame Lefschetz theorems for F-overconvergent isocrystals, \cite{AE16}] 
Let $X\hookrightarrow \bar X$ be a good regular projective compactification of a smooth quasi-projective scheme $X$,  geometrically irreducible  over a finite field,  such that $\bar X\setminus X$ is a normal crossings divisor with smooth components. Let $\bar C$ be a curve, smooth complete intersection of ample smooth ample divisors in good position with respect to 
$\bar X\setminus X$. 
Then  the restriction to $C$ of any $F$-overconvergent isocrystal irreducible $M$ is irreducible. 
\end{thm}
\noindent
One also  has the precise analog of the Lefschetz Theorem~\ref{thm:curve}.
\begin{thm}[Abe-Esnault \cite{AE16}] \label{thm:curve_crys}
Let $X$ be a smooth variety over $\F_q$ and $M$ be an irreducible $F$-overconvergent isocrystal. Then there is a smooth curve $C \to  X$ such that $V|_C$ is irreducible. One can request $C$ to pass through a finite number of closed points $x\in X$ with the same residue field $k(x)$. 

\end{thm}
\noindent
It is beyond the scope of these notes to give the essential points of the proof of these two theorems. We observe that one strongly uses the Tannakian structure of the category. Rather than proving the theorems as stated, one shows that the restriction to a good curve preserves the Tannakian group spanned by one object.  
To this aim, one ingredient is a version of class field theory for rank one $F$-overconvergent isocrystals, due to Abe. This enables one to argue purely cohomologically at the level of the global sections in the Tannakian category, and their behavior after restriction to a curve. This also yields a new proof of Theorem~\ref{thm:curve}.
\\[.1cm]
Theorem~\ref{thm:curve_crys} has a number of consequences.  The first one is 
\begin{thm}[Abe-Esnault, \cite{AE16}] \label{thm:AEcomp}
(2) is true. 
\end{thm}
\noindent
The existence of $V$ as a Weil sheaf, without the irreducibility property, has been proved independently by Kedlaya \cite{Ked16}, introducing weights, which are not discussed in these notes. \\[.2cm]
Other ones consist in transposing on the crystalline side what one knows on the $\ell$-adic side, such as Deligne's Finiteness Theorem~\ref{thm:finiteness}:  in bounded rank and bounded ramification, there are, up to twist by a rank one $F$-isocrystal of $\F_q$, only finitely many isomorphism classes of irreducible $F$-overconvergent isocrystals. The notion of bounded ramification here is not intrinsic to the crystalline theory. One says that the $F$-overconvergent isocrystal has ramification bounded by an effective Cartier divisor $D$ supported at infinity of $X$ if a $\sigma$-companion has. This notion does not depend on the choice of the isomorphism $\sigma: \bar \Q_\ell \xrightarrow{\cong} \bar \Q_p$ chosen (\cite{AE16}).\\[.2cm]

\noindent
There is a {\it hierarchy} of $F$-overconvergent isocrystals. \\[.1cm]
Among the $F$-overconvergent isocrystals, there are those which, while restricted to any closed point of $X,$ are  unit-root $F$-isocrystals. This simply means that they  consist of a finite dimensional vector space over the field of fraction of the Witt vectors of the residue field of the point, together with a $\sigma$-linear isomorphism with slopes equal to $0$. (We do not discuss slopes here).
This defines the category of {\it unit-root} F-isocrystals, as a subcategory of the category of $F$-overconvergent isocrystals. \\[.1cm]
Crew \cite{Cre87} proves that they admit a lattice, that is a crystal with the same isocrystal class, which is stabilized by the Frobenius action, which implies that the lattice  is locally free. 
Such a lattice is defined by a representation
$\pi_1(X,\bar x)\to GL(r, R),$ where $R$ is a finite extension of $\Z_p$. In particular, all eigenvalues of  the Frobenius at closed points are $p$-adic units in $\bar \Q_p$. Drinfeld defines an unit-root $F$-overconvergent isocrystal to be {\it absolute unit-root} if the image of  those eigenvalues by any automorphism of $\bar \Q_p$  are still $p$-adic units.

\begin{thm}[Koshikawa, \cite{Kos15}] \label{thm:koshikawa}
Irreducible absolute unit-root  $F$-overconvergent  isocrystals with finite determinant are iso-constant, that is the restriction of the representation to $\pi_1(\bar X,  x)$ has finite monodromy.

\end{thm}
\noindent
Isoconstancy means that after a finite \'etale cover, the isocrystal is constant, that is comes from an isocrystal on the ground field. 
\noindent

\begin{rmk}  One can summarize geometrically, as opposed to analytically,   the following different variants of isocrystals. 
\begin{enumerate}
\item 
Irreducible absolute $F$-overconvergent  unit-root isocrystals with finite determinant: those are the  iso-constant ones;
\item Irreducible  $F$-overconvergent  unit-root isocrystals: those are  the ones which are potentially unramified;
\item Irreducible  $F$-overconvergent isocrystals: those are the $F$-convergent isocrystals which become nilpotent after some alteration.

\end{enumerate}
The point (2), not discussed here, is used in the proof of (1) (Theorem~\ref{thm:koshikawa}) and is due to Tsusuki \cite{Tsu02}. 

\end{rmk}

\noindent
{\it Acknowledgements:}  It is a pleasure to thank Jakob Stix for a discussion on separable base points reflected in Section~\ref{s:pi}.  We thank Tomoyuki Abe and Atsushi Shiho for discussions. We thank the public of the Santal\'o lectures at the Universidad Complutense de  Madrid (October 2015) and the Rademacher lectures at the University of Pennsylvania (February 2016), where some points discussed in those notes were presented. 
In particular, we thank
Ching-Li Chai for an enlightening  discussion on compatible systems. We thank Moritz Kerz for discussions we had when we tried to understand Deligne's program in Weil II while writing \cite{EK12}. We thank the two referees for their friendly and thorough reports which helped us to  improve the initial version of these notes.

\end{document}